\documentclass[12pt]{amsart}
\input{xy}
\xyoption{all}
\usepackage{graphicx}
\usepackage{color}
\usepackage{verbatim}
\usepackage{mathrsfs} 
\usepackage{amsmath, amsthm, amssymb, amsfonts, amscd, graphicx, endnotes, tikz, color, hyperref, mathtools}
\usepackage{amsmath,amscd}

\newtheorem{theorem}{Theorem} [section]
\newtheorem{prop}[theorem]{Proposition}
\newtheorem{lemma}[theorem]{Lemma}

\theoremstyle{definition}

\theoremstyle{remark}
\newtheorem{remark}[theorem]{Remark}

\begin{document}
	\title[ Preperiodic Points with local rationality conditions]{ Preperiodic points with local rationality conditions in the quadratic unicritical family }
	\date{\today}
	\author{Chatchai Noytaptim}
	\address{Department of Mathematics,  Oregon State University; Corvallis OR, 97331 U.S.A.}
	\email{noytaptc@oregonstate.edu}	
	\maketitle
	\thispagestyle{empty}
	\begin{abstract} For rational numbers $c$, we present a trichotomy of the set of totally real (totally $p$-adic, respectively) preperiodic points for maps in the quadratic  unicritical family $f_c(x)=x^2+c$. As a consequence, we classify quadratic polynomials $f_c$ with rational parameters $c\in\mathbb{Q}$ so that $f_c$ has only finitely many totally real (totally $p$-adic, respectively) preperiodic points.  These results rely on an adelic Fekete-type theorem and dynamics of the filled Julia set of $f_c$. Moreover, using a numerical criterion introduced in \cite{NP}, we make explicit calculations of the set of totally real $f_c$-preperiodic points when $c=-1,0,\frac{1}{5}$ and $\frac{1}{4}.$ \end{abstract}
	\section{Introduction}
	Throughout this article, we study dynamics of a quadratic polynomial of the form $f_c(x)=x^2+c$ for $c\in \overline{\mathbb{Q}}.$ An algebraic number $\alpha\in\overline{\mathbb{Q}}$ is called a preperiodic point of $f_c$ if it satisfies $f_c^n(\alpha)=f_c^m(\alpha)\,$ for some integers $0\leq m<n$, where $f_c^n:=f_c\circ...\circ f_c$ denotes the $n$-fold decomposition of $f_c$ with itself. In other words, the (forward) orbit of $\alpha$ under the iteration of $f_c$ is finite. Denote by $\mathrm{PrePer}(f_c)$ the set of all preperiodic points of $f_c$ in $\overline{\mathbb{Q}}.$\\
	
	\indent Recall that  an algebraic number $\alpha$ is totally real if its minimal polynomial over $\mathbb{Q}$ splits completely over $\mathbb{R}$, or equivalently, if every complex embedding $\mathbb{Q}(\alpha)\hookrightarrow \mathbb{C}$ has image contained in $\mathbb{R}$. Let  $\mathbb{Q}^{\mathrm{tr}}$ denote the collection of all totally real algebraic numbers.  One may wonder, for which rational parameters $c\in\mathbb{Q}$ is  the set of totally real $f_c$-preperiodic points $\mathrm{PrePer}(f_c)\cap \mathbb{Q}^{\mathrm{tr}}$  finite.  This question leads to the following result.
	\begin{theorem} \label{mainresult} Let $c$ be a rational number and let $f_c(x)=x^2+c.$ 
		\begin{enumerate}
			\item[(a)] If $c>\frac{1}{4}$, then $\mathrm{PrePer}(f_c)\cap \mathbb{Q}^{\mathrm{tr}}=\emptyset$.
			\item[(b)] If $-2<c\leq \frac{1}{4}$, then $\mathrm{PrePer}(f_c)\cap \mathbb{Q}^{\mathrm{tr}}$ is nonempty and finite.
			\item[(c)] If $c\leq -2$ then $\mathrm{PrePer}(f_c)\subseteq \mathbb{Q}^{\mathrm{tr}}$; in particular, $\mathrm{PrePer}(f_c)\cap \mathbb{Q}^{\mathrm{tr}}$  is infinite.
		\end{enumerate}
	\end{theorem}
	The idea behind the proof of  Theorem \ref{mainresult}, particularly part (b), is an adelic version of Fekete's theorem;  for example,  Baker-Hsia \cite[Proposition 6.1]{BH}, Cantor \cite[Theorem 5.1.2]{Can}, and Rumely \cite[Theorem 6.3.1]{Rum}. We defer details until section $4$.  \\ 
	
	\indent For some rational numbers $c\in (-2,\frac{1}{4}]$, we are able to compute totally real preperiodic points of $f_c$ explicitly by using the technique in \cite{NP}(which concerns totally real algebraic integers with all Galois conjugates lying in an interval of length less than $4$). However, for arbitrary $c\in\mathbb{Q}$ the preperiodic points of $f_c$ are not always algebraic integers. Thus we need to  get around this issue by making an adequate change of coordinates to $f^{\phi}_c$ by an affine automorphism $\phi$. This gives $\mathrm{PrePer}(f^{\phi}_c)\subseteq \overline{\mathbb{Z}}$, see Lemma \ref{coorchangetoalgint}. 
	 Hence we have the following full description the set $\mathrm{PrePer}(f_c)\cap \mathbb{Q}^{\mathrm{tr}}$ where $c\in \{-1,0,\frac{1}{5}, \frac{1}{4}\}$.
	\begin{theorem}It holds that
	\begin{enumerate}
		\item[(a)] $\mathrm{PrePer}(x^2)\cap \mathbb{Q}^{\mathrm{tr}}=\{\pm 1,0\}$.
		\item[(b)] $\mathrm{PrePer}\left(x^2+\frac{1}{4}\right)\cap \mathbb{Q}^{\mathrm{tr}}=\left\{\pm\frac{1}{2}\right\}$.
		\item[(c)] $\mathrm{PrePer}(x^2-1)\cap \mathbb{Q}^{\mathrm{tr}}=\left\{0,\pm1,\pm\sqrt{2}, \frac{-1\pm\sqrt{5}}{2}, \frac{1\pm\sqrt{5}}{2}\right\}$.
		\item[(d)] $\mathrm{PrePer}\left(x^2+\frac{1}{5}\right)\cap \mathbb{Q}^{\mathrm{tr}}=\left\{ \frac{-1\pm\sqrt{5}}{2\sqrt{5}}, \frac{1\pm\sqrt{5}}{2\sqrt{5}}\right\}$.
	\end{enumerate}
\end{theorem}
	Note that the technique presented in \cite{NP} is not limited to rational parameters $c\in \{-1,0,\frac{1}{5}, \frac{1}{4}\}$, but for other parameters $c\in (-2,\frac{1}{4}]\cap \mathbb{Q}$ the computation becomes more difficult.  It would also be interesting to give a bound on $\mathrm{PrePer}(f_c)\cap \mathbb{Q}^{\mathrm{tr}}$ for additional parameters $c$, but we do not pursue it here.\\
	
	In a similar spirit to Theorem \ref{mainresult}, we next provide a trichotomy result for  the set $\mathrm{PrePer}(f_c)\cap \mathbb{Q}^{\mathrm{tp}}$, where  $\mathbb{Q}^{\mathrm{tp}}$ denotes the collection of all totally $p$-adic algebraic numbers. Recall an algebraic number $\alpha$ is totally $p$-adic if its minimal polynomial over $\mathbb{Q}$ splits completely over $\mathbb{Q}_p$, or equivalently, if every  embedding $\mathbb{Q}(\alpha)\hookrightarrow \mathbb{C}_p$ has image contained in $\mathbb{Q}_p$. 
	\begin{theorem} \label{maintotpadic} Fix an odd prime $p$. Let $c\in\mathbb{Q}$ and  $f_c(x)=x^2+c$.   
		\begin{enumerate}
			\item[(a)] If $|c|_p\leq 1$, then $\mathrm{PrePer}(f_c)\cap \mathbb{Q}^{\mathrm{tp}}$ is  finite.
			\item[(b)] If $|c|_p>1$  and $-c$ is not a square in $\mathbb{Q}_p$, then $\mathrm{PrePer}(f_c)\cap \mathbb{Q}^{\mathrm{tp}}$ is empty.
			\item[(c)] If $|c|_p>1$  and $-c$ is a square in $\mathbb{Q}_p$, then $\mathrm{PrePer}(f_c)\subset \mathbb{Q}^{\mathrm{tp}}$; in particular, $\mathrm{PrePer}(f_c)$ contains infinitely many totally $p$-adic $f_c$-preperiodic points.
		\end{enumerate}
	\end{theorem}
Again, the first part of Theorem  \ref{maintotpadic} is an application of an adelic Fekete theorem. Whereas part (b) and part (c) are consequences of the shape of $p$-adic filled Julia set of $f_c$ due to Benedetto-Briend-Perdry \cite[Th$\acute{\text{e}}$or$\grave{\text{e}}$me 1.1]{BBP}.  In contrast to the totally real situation, the set $\mathrm{PrePer}(f_c)\cap \mathbb{Q}^{\mathrm{tp}}$  is not easily computed quantitatively.\\

	\indent This article is organized as follows. We study dynamics of $f_c$ over archimedean and nonarchimedean fields  in $\S 2$.  We then  provide  the trichotomy of the set $\mathrm{PrePer}(f_c)\cap \mathbb{Q}^{\mathrm{tr}}$ and the set  $\mathrm{PrePer}(f_c)\cap \mathbb{Q}^{\mathrm{tp}}$  in $\S 3$ and $\S 5$, respectively. In $\S 4$, we compute the set $\mathrm{PrePer}(f_c)\cap \mathbb{Q}^{\mathrm{tr}}$ for certain rational  parameters $c\in (-2,\frac{1}{4}]$.   \\
	
	\noindent \textbf{Acknowledgement.} The author would like to thank his PhD advisor, Clayton Petsche, for several fruitful conversations and guidance over the course of  this project and useful suggestions  concerning the exposition.
	\section{Preliminaries}
	\indent In what follows, $L$ is an arbitrary field which is complete with respect to an absolute value $|\cdot|$. If $|\cdot|$ is an archimedean absolute value, we say that $L$ is an archimedean field (e.g., $L$ is either $\mathbb{R}$ or $\mathbb{C}$). Otherwise, we call $L$ a nonarchimedean field (e.g., $\mathbb{C}_p$ or $\mathbb{Q}_p$  endowed with $|\cdot|_p$ ).
Given a polynomial $f(x)\in L[x]$ of degree at least $2$. Then the corresponding filled Julia set of $f$ is defined as 
$$\mathcal{K}_{L}(f)=\{x\in L\mid  \text{the forward orbit of \,\,$x$\,\,is bounded with respect to}\,\,|\cdot|\}.$$
We present a trichotomy of the shape of the filled Julia set of $f_c=x^2+c$ over an archimedean valued field $(L,|\cdot|)$.
	\begin{prop} \label{filledJuliatrichotomy} Let $c$ be a real number.
		\begin{enumerate}
			\item[(i)]  If $c>\frac{1}{4}$, then $\mathcal{K}_{\mathbb{R}}(f_c)=\emptyset$.
		\end{enumerate}
		Let $a_c=\frac{1+\sqrt{1-4c}}{2}$ when $c\leq \frac{1}{4}$. 
				\begin{enumerate}
			\item[(ii)] If  $-2\leq c\leq \frac{1}{4}$, then $\mathcal{K}_{\mathbb{R}}(f_c)=[-a_c,a_c]$.
			\item[(iii)] If  $c\leq -2$, then $\mathcal{K}_{\mathbb{C}}(f_c)\subseteq [-a_c,a_c]$.
		\end{enumerate} 
	\end{prop}
\begin{remark} It is worthwhile to point out  that the equality $\mathcal{K}_{\mathbb{C}}(f_c)=[-a_c,a_c]$ in (iii) occurs only when $c=-2$. That is, the complex filled Julia set of the Chebyshev polynomial $f_{-2}(x)=x^2-2$ is the closed real interval $[-2,2]$ and it coincides with the Julia  set of $f_{-2}(x).$
\end{remark}
	\begin{proof} Let $x$ be a real number. Since $c>\frac{1}{4}$, there is $\varepsilon>0$ so that $c=\frac{1}{4}+\varepsilon$. Thus $$f_c(x)=x^2+c=x^2-x+\frac{1}{4}+x+\varepsilon=\left(x-\frac{1}{2}\right)^2+x+\varepsilon>x+\varepsilon.$$ Inductively, we obtain that $$f^n_c(x)>f^{n-1}_c(x)+\varepsilon>x+(n-1)\varepsilon+\varepsilon=x+n\varepsilon.$$ Letting $n\rightarrow\infty$, it means that the orbit of $x$ under $f_c$ is unbounded. The proof of (i) is complete. \\
		\indent In part (ii), we assert that $\mathcal{K}_{\mathbb{R}}(f_c)=[-a_c,a_c]$ where $a_c=\frac{1+\sqrt{1-4c}}{2}$ is one of the fixed points of $f_c$.  First, we verify that $\mathcal{K}_{\mathbb{R}}(f_c)\subseteq [-a_c,a_c]$. For $-2\leq c\leq \frac{1}{4}$, it is easy to compute that $\frac{1}{2}\leq a_c\leq 2$. Let  $x_0>a_c\geq \frac{1}{2}$.  Then we may consider the polynomial  $g(x)=f_c(x)-x$ in which  $g'(x)=2x-1>0$. That is, $g$ is increasing on $(a_c,+\infty).$ Let $\delta=x_0-a$ and so $x_0=a_c+\delta$. For each $x\geq x_0$, 
it follows that
			\begin{align*}
				g(x)>g(a_c+\delta)
				&\Longrightarrow f_c(x)-x>a_c^2+2a_c\delta+\delta^2+c-a_c-\delta\\
				&\Longrightarrow f_c(x)-x>\delta^2 \quad (\text{using}\,\,a_c^2+c=a_c\,\,\text{and}\,\,2a_c\geq 1)\\
				&\Longrightarrow f_c(x)>x+\delta^2
			\end{align*}
The iteration of $f_c$ implies that  $f_c^n(x)\geq x_0+n\delta^2\rightarrow+\infty$ as $n\rightarrow+\infty$.
	In the case $x<-a_c$, we exploit the fact that $f_c$ is even  to deduce that again $f^n_c(x)\rightarrow +\infty$ as $n\rightarrow+\infty$. We conclude that $\mathcal{K}_{\mathbb{R}}(f_c)\subseteq [-a_c,a_c]$.
	
		 For the reverse inclusion $[-a_c, a_c]\subseteq \mathcal{K}_{\mathbb{R}}(f_c)$. First  recall $\frac{1}{2}\leq a_c\leq 2$ for $-2\leq c\leq \frac{1}{4}$. This gives us $-1\leq 1-a_c\leq \frac{1}{2}$ and so $-a_c\leq a_c-a_c^2=c$. Suppose that $x\in [-a_c,a_c]$. Thus it follows that $$c\leq x^2+c\leq a_c^2+c=a_c.$$ Equivalently, $f_c([-a_c,a_c])\subseteq [c,a_c]$. Since we know that $-a_c\leq c$,  this immediately implies that $f_c([-a_c,a_c])\subseteq [-a_c,a_c]$ and hence, inductively, it yields $f_c^n([-a_c,a_c])\subseteq [-a_c,a_c]$ for all $n\geq 1$. This proves $[-a_c,a_c]\subseteq \mathcal{K}_{\mathbb{R}}(f_c)$.
	\\ \indent To prove (iii), 
			we follow the idea  presented in \cite[Lemma 7.1]{Miln} and we include details here for completeness. 
			We first assert that the complex Julia set of $f_c$ is contained in the real line when $c\leq -2$. To see this, we consider the  real interval $[-a_c,a_c]$. It is not hard to see that solutions of equation $f_c(x)=w$ for $w\in [-a_c,a_c]$ are again in $[-a_c,a_c]$.  This means that $$f^{-n}_c([-a_c,a_c])\subseteq [-a_c,a_c]$$ inductively, for all $n\geq 1.$ Notice that  $[-a_c,a_c]$ contains a (real) repelling fixed point  $a_c$ of $f_c$ because the multiplier $|f'_c(a_c)|=2a_c>1$ when $c\leq -2$, see \cite[Definition 4.5]{Miln}. It means that $a_c$ is in the Julia set $\mathcal{J}(f_c)$ of $f_c.$
			  Hence $[-a_c,a_c]$ must contain the entire Julia set $\mathcal{J}(f_c)$ because the iterated preimages of $a_c$ are dense in $\mathcal{J}(f_c)$, see \cite[Corollary 4.13]{Miln}. In other words, 
			  $$\mathcal{J}(f_c)=\overline{\{x\in\mathbb{R}\mid f_c^n(x)=a_c,\,\,n\geq0\}}\subseteq [-a_c,a_c].$$ It follows that, for any real number $c\leq -2$, the Julia set $\mathcal{J}(f_c)$ must coincide with the filled Julia set  $\mathcal{K}_{\mathbb{C}}(f_c)$ because $\mathcal{J}(f_c)=\partial \mathcal{K}_{\mathbb{C}}(f_c)$ (i.e., the Julia set of polynomial map  $f_c$  is the topological boundary of the filled Julia set of $f_c$) and $\mathcal{K}_{\mathbb{C}}(f_c)$ has empty interior. This follows from the fact that a closed set with empty interior is its own boundary, for example, see \cite[Remark 9.3]{Gil}. Consequently, the filled Julia set $\mathcal{K}_{\mathbb{C}}(f_c)\subseteq [-a_c,a_c]$ when $c\leq -2$.

	\end{proof}
	When $L$ is a nonarchimedean field, one obtains the following shapes of nonarchimedean filled Julia set of $f_c(x)=x^2+c\in L[x]$. This is also proved in \cite{BBP}, but we include the proof here for completeness.
	\begin{prop}\label{padicfjs}   Let $L$ be a complete field equipped with a nonarchimedean absolute value $|\cdot|$. Then the filled Julia set $\mathcal{K}_L(f_c)$ has the following dichotomy : 
		\begin{enumerate}
			\item[(i)] If $|c|\leq 1$, then $\mathcal{K}_L(f_c)=\{x\in L\mid |x|\leq 1\}$.
			\item[(ii)] If $|c|>1$, then $\mathcal{K}_L(f_c)\subseteq \{x\in L \mid |x|^2=|c|\}$.
		\end{enumerate}
	\end{prop}
	\begin{proof} 
		Suppose that $|c|\leq 1$. If $|x|\leq 1$, then we have by the strong triangle inequality that
			$|f_c(x)|=|x^2+c|\leq \max\{|x|^2,|c|\}\leq 1.$ Inductively, we can verify, for all $n\geq 1$, that $|f^n_c(x)|\leq 1$. Hence $x\in \mathcal{K}_L(f_c)$ and it gives the inclusion $\{x\in L\mid  |x|\leq 1\}\subseteq \mathcal{K}_L(f_c).$ If $|x|>1$, it implies
			$|f_c(x)|=\max\{|x|^2,|c|\}=|x|^2>1$ and computing the iterations of $x$, by induction, for all $n\geq 1$, yields $$|f^n_c(x)|=|x|^{2^n}.$$ Since $|x|>1$, it follows that the orbit of $x$ under applying the map $f_c$ grows without bounds. Therefore $x$ does not belong to $\mathcal{K}_L(f_c)$.  This completes part (i).\\
		\indent To verify part (ii), suppose that $|c|>1$. To obtain the desired inclusion, it suffices to show that if $x\in L$ satisfies $|x|^2\neq |c|$, then $x\not\in \mathcal{K}_L(f_c)$. If $|x|^2>|c|$, it follows from the strong triangle inequality that  $$|f_c(x)|=|x|^2>1.$$ Inductively,  $|f^n_c(x)|=|x|^{2^n}$ tends to $+\infty$ as $n$ tends to $+\infty$. Similarly, we can show that $|f^n_c(x)|$ is unbounded as $n\rightarrow \infty$ for $|x|^2<|c|$. In fact, by the strong triangle inequality we have $|f_c(x)|=|c|$ and $|f^2(x)|=|c|^2>|c|$. The conclusion follows from the first case when $|x|^2>|c|.$ 
	\end{proof}
	
	\section{Totally real preperiodic points} 
	We review necessary background in archimedean and nonarchimedean potential theory. Let $K$ be a number field.  Then $K_v$ is the completion of $K$ at each place $v\in M_K$. Denote by $\mathbb{C}_v$  the completion of an algebraic closure $\overline{K}_v$ of $K_v$. An adelic set  $\mathbb{E}=\prod_v E_v$ is a collection of bounded sets $\{E_v\}_{v\in M_K}$, where $E_v=\mathcal{O}_v$ (ring of integers of $\mathbb{C}_v$) for almost all $v$.  In this article, our absolute values $|\cdot|_v$ coincide with usual real and $p$-adic absolute values when restricted to $\mathbb{Q}$ and this normalization is the same as described in \cite[ $\S2$]{BH}.   Then we define the $n$-diameter of $E_v$ at each place $v$ of $K$ to be 
	$$d_n(E_v)=\sup_{x_1,...,x_n\in E_v}\prod_{i\neq j}|x_i-x_j|_v^{\frac{1}{n(n-1)}}.$$ The quantity $d_n(E_v)$ describes the supremum of geometric mean of pairwise distance among $n$-point in the bounded set $E_v$. 
	This sequence $\{d_n(E_v)\}_{n\geq 2}$ is monotone decreasing and the limit exists, see \cite[$\S 4$]{BH}. The limit  is known as the transfinite diameter of $E_v$
	$$d_{\infty}(E_v)=\lim_{n\rightarrow\infty}d_n(E_v).$$  Thus the transfinite diameter of an adelic set $\mathbb{E}=\prod_v E_v$ is given by 
	$$d_{\infty}(\mathbb{E}):=\prod_v d_{\infty}(E_v)^{N_v}$$ where the numbers $N_v=[K_v:\mathbb{Q}_v]\geq 1$ are local degrees. Note that the infinite product is indeed finite because almost all $d_{\infty}(E_v)=1$, see \cite[Example 5.2.16]{Rum}.\\
	
	\noindent For each place $v$ of $K$, the $v$-adic filled Julia set of $f_c$, denoted by $\mathcal{K}_{\mathbb{C}_v}(f_c)$, is 
	$$\mathcal{K}_{\mathbb{C}_v}(f_c):=\{ x\in \mathbb{C}_v \mid  \text{the forward $f_c$-orbit of \,$x\,$ is bounded with respect to}\,\,|\cdot|_v\}.$$   With this notation, Baker and Hsia \cite[Theorem 4.1]{BH} showed that the transfinite  diameter of $v$-adic filled Julia set of the monic polynomial $f_c$ is 
	\begin{equation}d_{\infty}(\mathcal{K}_{\mathbb{C}_v}(f_c))=1\label{capvadicfilled}\end{equation} for all $v\in M_K.$ 
	The capacity of many bounded sets over $\mathbb{C}$ are known. For example, Ransford \cite[Corollary 5.4]{Rans} shows that the capacity of a real segment is one-quarter of its length. \\
	
	\indent The following adelic Fekete-type theorem is due to Baker-Hsia \cite[Proposition 6.1]{BH} (cf. Cantor \cite[Theorem 5.2.1]{Can} and Rumely \cite[Theorem 6.3.1]{Rum}).
	\begin{prop} \label{adelicfekete} Let $K$ be a number field and let $\mathbb{E}=\prod_v E_v$ be an adelic set so that $d_{\infty}(\mathbb{E})<1$. Then there are only finitely many algebraic numbers having all Galois conjugates contained in $E_v$ for all $v\in M_K$.
	\end{prop}
	We are now  ready to prove our main result for this section, and it is restated for the reader's convenience.
	\begin{theorem} \label{mainmain} Let $c$ be a rational number and let $f_c(x)=x^2+c.$
		\begin{enumerate}
			\item[(a)] If $c>\frac{1}{4}$, then $\mathrm{PrePer}(f_c)\cap \mathbb{Q}^{\mathrm{tr}}$ is empty.
			\item[(b)] If $-2<c\leq \frac{1}{4}$, then $\mathrm{PrePer}(f_c)\cap \mathbb{Q}^{\mathrm{tr}}$ is nonempty and finite.
			\item[(c)] If $c\leq -2$ then $\mathrm{PrePer}(f_c)\subseteq \mathbb{Q}^{\mathrm{tr}}$; in particular, $\mathrm{PrePer}(f_c)\cap \mathbb{Q}^{\mathrm{tr}}$ is infinite.
		\end{enumerate}
	\end{theorem}
	\begin{proof}
		Part (a)  follows immediately from Proposition \ref{filledJuliatrichotomy}(i). 
	 For each $c\in (-2,\frac{1}{4}]$, totally real preperiodic points of $f_c$ are algebraic numbers and together with all of their Galois conjugates belonging to the interval $[-a_c,a_c]$ where $a_c=\frac{1}{2}+\frac{\sqrt{1-4c}}{2}$ is the quantity defined in Proposition \ref{filledJuliatrichotomy}.  Consider the  adelic set $\mathbb{E}=\{E_v\}_v$ where $$E_v=\begin{cases}\mathcal{K}_{\mathbb{R}}(f_c)&\mbox{if}\,\,v=\infty\\ \mathcal{K}_{\mathbb{C}_v}(f_c)&\mbox{if}\,\,v\neq \infty\end{cases}$$ The transfinite diameter of the adelic  set $\mathbb{E}$ is 
			$$d_{\infty}(\mathbb{E})=d_{\infty}([-a_c,a_c])\prod_{v\neq\infty }d_{\infty}(\mathcal{K}_{\mathbb{C}_v}(f_c))=\frac{a_c}{2}\prod_{v\neq\infty }d_{\infty}(\mathcal{K}_{\mathbb{C}_v}(f_c)).$$ Note that  $d_{\infty}(\mathcal{K}_{\mathbb{C}_v}(f_c))=1$ at all finite places $v$ by (\ref{capvadicfilled}). Since $\frac{a_c}{2}<1$ when $c\in (-2,\frac{1}{4}]$, this means that $d_{\infty}(\mathbb{E})<1$. Applying Proposition \ref{adelicfekete}, there are only finitely many totally real algebraic numbers having all of its conjugates in $[-a_c,a_c]$ at the place $v=\infty$ and $\mathcal{K}_{\mathbb{C}_v}(f_c)$ for all finite places $v.$ Hence the finiteness of $\mathrm{PrePer}(f_c)\cap \mathbb{Q}^{\mathrm{tr}}$ follows. \\
		\indent It remains to verify (c). For $c\leq -2$, it can be deduced from Proposition \ref{filledJuliatrichotomy}(iii) that all preperiodic points of $f_c$ are totally real algebraic numbers. That is, it is true because if $\alpha$ is $f_c$-preperiodic then so are all of its conjugates, and  $$\mathrm{PrePer}(f_c)\subseteq \mathcal{K}_{\mathbb{C}}(f_c)\subseteq \mathbb{R}.$$ 
		Therefore $\mathrm{PrePer}(f_c)\cap \mathbb{Q}^{\mathrm{tr}}$ is infinite.
	\end{proof}
\section{Quantitative Calculations}
  Let $L$ be a complete field with respect to an absolute value $|\cdot|$. 
Recall the  group of affine automorphisms $\phi : \mathbb{A}^1\rightarrow\mathbb{A}^1$ defined over $L$ is 
	$$\mathrm{Aff}(L):=\{ax+b\mid  a,b\in L,\,\,a\neq0\}.$$
For each  $f(x)\in L[x]$ and $\phi \in\mathrm{Aff}(L)$, we define $f^{\phi}:=\phi^{-1}\circ f\circ \phi$.  Then we say $f$ and $g$ are conjugate if $g=f^{\phi}$ for some $\phi\in \mathrm{Aff}(L).$

\noindent Recall the filled Julia set of $f(x)\in L[x]$ of degree $\geq 2$  is 
$$\mathcal{K}_{L}(f)=\{x\in L\mid  \text{the forward orbit of \,\,$x$\,\,is bounded with respect to}\,\,|\cdot|\}.$$ The following proposition is straightforward and follows from the above definitions.
\begin{prop} \label{changeofcoordinates} Let  $\phi\in \mathrm{Aff}(L)$ and $f(x)\in L[x]$. Then the following statements hold :
	\begin{enumerate}\item[(i)] $\alpha$ is a preperiodic point of $f$ if and only if $\phi^{-1}(\alpha)$ is a preperiodic point of $f^{\phi}.$ 
		\item[(ii)]  $\mathcal{K}_L(f^{\phi})=\phi^{-1}(\mathcal{K}_L(f)).$
	\end{enumerate}
\end{prop}
\indent Note that preperiodic points of $f_c$ when $c\in\mathbb{Q}$ are not always algebraic integers.  Thus we need to make a suitable change of coordinates of $f_c$ to $f^{\phi}_c$ for some $\phi\in \mathrm{Aut}(\mathbb{C})$ so that all preperiodic points of $f^{\phi}_c$ are algebraic integers due to the following Lemma.
	\begin{lemma}\label{coorchangetoalgint} Let $a$ and $b$ be relatively prime integers and $b\neq 0$. Let $f(x)=x^2+\frac{a}{b}$ and $\phi(x)=b^{-1/2}x$. Then all preperiodic points for $f^{\phi}(x)=b^{-1/2}(x^2+a)$ are algebraic integers.
\end{lemma}
\begin{proof} Let $p$ be any rational prime number and let $(\mathbb{C}_p,|\cdot|)$ be a complete nonarchimedean field with respect to $|\cdot|.$    We want to show that all preperiodic points of $f_c^{\phi}$ are in the ring of integer $\mathcal{O}_p$. If $p$ does not divide $b$, then we have $\phi^{-1}(\mathcal{O}_p)=\mathcal{O}_p$ and \begin{align*}\mathcal{K}_{\mathbb{C}_p}(f^{\phi})&=\phi^{-1}(\mathcal{K}_{\mathbb{C}_p}(f_c))\quad (\text{using  Proposition\,\,\,\ref{changeofcoordinates}(ii)})\\&=\phi^{-1}(\mathcal{O}_p)\quad (\text{using Proposition\,\,\,\ref{padicfjs}(i)})\\&=\mathcal{O}_p.\end{align*}  It remains to treat the case when $p$  divides $b$. Then we have
	\begin{align*}
 \mathcal{K}_{\mathbb{C}_p}(f^{\phi})&=\phi^{-1}(\mathcal{K}_{\mathbb{C}_p}(f))\quad (\text{using Proposition\,\,\,\ref{changeofcoordinates}(ii)})\\
		&\subseteq \phi^{-1}(\{x\in \mathbb{C}_p \mid |x|=|b|^{-1/2}\})\quad (\text{using Proposition\,\,\,\ref{padicfjs}(ii)})\\
		&\subseteq \{x\in \mathbb{C}_p \mid |b^{-1/2}x|=|b|^{-1/2}\}\quad(\text{by definition of}\,\,\,\phi)\\
		&\subseteq \mathcal{O}_p
	\end{align*}
It follows that $\mathcal{K}_{\mathbb{C}_p}(f_c)\subseteq \mathcal{O}_p$ for all primes $p$.
	Therefore every $\alpha\in \mathrm{PrePer}(f_c^{\phi})$ is an algebraic integer.
\end{proof}
Recall that theorem 8 in  \cite{NP} provides a convenient method to bound the degree of algebraic integers having all of its conjugates lying in an interval of length less than $4$. This tool is useful in our computation because we only have to search for low degree algebraic integers in the interval.
With the strategy described, we are able to calculate the set $\mathrm{PrePer}(f_c)\cap \mathbb{Q}^{\mathrm{tr}}$ for $c\in \{-1,0,\frac{1}{5}, \frac{1}{4}\}$ as follows. 
	\begin{theorem} \label{quantitativecomp} It holds that
		\begin{enumerate}
			\item[(a)] $\mathrm{PrePer}(x^2)\cap \mathbb{Q}^{\mathrm{tr}}=\{\pm 1,0\}$.
			\item[(b)] $\mathrm{PrePer}\left(x^2+\frac{1}{4}\right)\cap \mathbb{Q}^{\mathrm{tr}}=\left\{\pm\frac{1}{2}\right\}$.
			\item[(c)] $\mathrm{PrePer}(x^2-1)\cap \mathbb{Q}^{\mathrm{tr}}=\left\{0,\pm1,\pm\sqrt{2}, \frac{-1\pm\sqrt{5}}{2}, \frac{1\pm\sqrt{5}}{2}\right\}$.
			\item[(d)] $\mathrm{PrePer}\left(x^2+\frac{1}{5}\right)\cap \mathbb{Q}^{\mathrm{tr}}=\left\{ \frac{-1\pm\sqrt{5}}{2\sqrt{5}}, \frac{1\pm\sqrt{5}}{2\sqrt{5}}\right\}$.
		\end{enumerate}
	\end{theorem}
\begin{remark} It is interesting to point out that  parameters in the set $\{-1,0,\frac{1}{5},\frac{1}{4}\}$ are not the only rational numbers for which the corresponding real part of the filled Julia set of the conjugate map $f^{\phi}_c$ has length less than $4$. The rational parameters $c=-\frac{1}{2}$ and $c=\frac{1}{6}$ also have this property, but the calculation in searching algebraic integers in $\mathcal{K}(f^{\phi}_c)\cap\mathbb{R}$ becomes more challenging. More precisely, \cite[Theorem 8]{NP} asserts that the real segment $\mathcal{K}(f^{\phi}_c)\cap\mathbb{R}$ contains algebraic integers of degree at most $97$ when $c\in\{-\frac{1}{2},\frac{1}{6}\}.$ For additional rational parameters in $(-2,\frac{1}{4}]$, the length of $\mathcal{K}(f^{\phi}_c)\cap\mathbb{R}$ is either $4$ or greater than $4$ and we do not have a known tool to explicitly calculate  $\mathrm{PrePer}(f_c)\cap \mathbb{Q}^{\mathrm{tr}}$ in this situation.
\end{remark}
	\begin{proof}
	Part (a) is simple. The only preperiodic points of the squaring map are $0$ and roots of unity. It is also clear that $\pm1$ are the only totally real roots of unity. \\
	\indent In part (b), let $\phi(x)=\frac{1}{2}x$. As in Lemma \ref{coorchangetoalgint}, we work with the conjugate map $f^{\phi}(x)=\frac{1}{2}(x^2+1)$ of $f(x):=f_{1/4}(x)=x^2+\frac{1}{4}.$ Thus the real segment of the filled Julia set of $f^{\phi}$ is $[-1,1]$ by Proposition \ref{changeofcoordinates}(ii) and totally real preperiodic points of $f^{\phi}$ belonging to this interval. Moreover, all preperiodic points of $f^{\phi}$ are algebraic integers by Lemma \ref{coorchangetoalgint}. Then applying Corollary $9$ in \cite{NP} (which states that algebraic integers in an interval of length $<\sqrt{5}$ are integers), it follows that $-1,0,$ and $1$ are the only algebraic integers all of whose conjugates lie in the interval  $[-1,1]$. But, $0$ is not preperiodic for $f^{\phi}$ because it does not belong to the $2$-adic filled Julia set of $f^{\phi}$. This means the only totally real preperiodic points of $f^{\phi}$ are $\pm1.$ Therefore $\pm\frac{1}{2}$ are the only totally real preperiodic points of $f$ by Proposition \ref{changeofcoordinates}(i).\\
 \indent To compute  part (c), we let $I$ denote the real segment of the filled Julia set of $g(x):=f_{-1}(x)=x^2-1$. Then 
			$$I:=\left[-\frac{1+\sqrt{5}}{2},\frac{1+\sqrt{5}}{2}\right].$$
			Let $\theta$ be  a totally real preperiodic point of $g$. We want to show that $[\mathbb{Q}(\theta):\mathbb{Q}]<7$. To see this, we apply Theorem 8 in\cite{NP} with 
			\begin{align*}
				a_n&=d_n\left(\left[-\frac{1+\sqrt{5}}{2},\frac{1+\sqrt{5}}{2}\right]\right)^{n(n-1)}=(1+\sqrt{5})^{n(n-1)}D_n\\
				b_n&=\frac{n^{2n}}{n!^2}
			\end{align*}
			The sequence $\{D_n\}$ is related to Jacobi polynomials and defined recursively in \cite[Theorem 2]{NP}. Then we calculate

			\begin{center}\begin{tabular}{ llll}
					$D_2=1$&$D_3=\frac{1}{16}$&$D_4=\frac{1}{3125}$&$D_5=\frac{27}{210827008}$\\
					$D_6=\frac{1}{259356749904}$&$D_7=\frac{3125}{368036388053733408768}$&$D_8=\frac{3125}{2333172671504650870750167741}$
			\end{tabular}\end{center}
	and thus 
			\begin{center}\begin{tabular}{ ll}
			$a_2=(1+\sqrt{5})^{2}D_2\approx 10.472...$&  $a_6=(1+\sqrt{5})^{30}D_6\approx 7702.496...$\\
			$a_3=(1+\sqrt{5})^{6}D_3\approx 71.777... $& $a_7=(1+\sqrt{5})^{42}D_7\approx22371.750...$\\
			$a_4=(1+\sqrt{5})^{12}D_4\approx 422.047...$& $a_8=(1+\sqrt{5})^{56}D_8\approx 48740.586...$\\
			$a_5=(1+\sqrt{5})^{20}D_5\approx2031.375...$
	\end{tabular}\end{center}
			Also, we compute 
			\begin{center}\begin{tabular}{ llll}
					$b_2=4$&$b_3=20.25$&$b_4\approx 113.777$&$b_5\approx 678.168$\\
					$b_6=4199.04$&$b_7\approx 26700.013$&$b_8\approx 173140.530$
			\end{tabular}\end{center}
			This yields $a_7<b_7$ and 
			$$\frac{a_8}{a_7}\approx 2.178...<\frac{b_8}{b_7}\approx6.484...$$ 
			Thus Theorem 8 in \cite{NP} is applied with $n_0=7$ and so $[\mathbb{Q}(\theta): \mathbb{Q}]<7$.    Using  Robinson's result \cite{Rob} (in which he classified, up to degree $8$, algebraic integers having all Galois conjugates containing in an interval of length less than $4$), it is impossible that $3\leq [\mathbb{Q}(\theta):\mathbb{Q}]\leq 6$ because there is a zero of the minimal polynomial satisfied by $\theta$ that does not belong to $I$. Thus the real interval $I$ contains algebraic integers of degree at most $2$. Again, using Table $1$ in \cite{Rob}, we obtain algebraic integers of degree $\leq 2$ having all of its conjugates in $I$ and so
			\begin{equation}\mathrm{PrePer}(g)\cap \mathbb{Q}^{\mathrm{tr}}\subseteq \left\{0,\pm1,\pm\sqrt{2}, \frac{-1\pm\sqrt{5}}{2}, \frac{1\pm\sqrt{5}}{2}\right\}.\label{preperbasilica}\end{equation} It is elementary to check that  $9$ algebraic integers in the set on the right side of (\ref{preperbasilica}) are $g$-preperiodic.  This finishes the computation in part (c).\\
	\indent The idea in part (d) is (b) and (c) combined.  Let $\phi(x)=\frac{1}{\sqrt{5}}x$. After making a suitable change of coordinates of $h(x):=f_{1/5}(x)=x^2+\frac{1}{5}$ to $h^{\phi}(x)=\frac{1}{\sqrt{5}}(x^2+1)$, we have the real segment of the filled Julia set of $h^{\phi}$ is $I$ (as denoted in part (c)) by Proposition \ref{changeofcoordinates}(ii). Thus the same argument shown in part (c) yields 
	$$\mathrm{PrePer}(h^{\phi})\cap \mathbb{Q}^{\mathrm{tr}}\subseteq \left\{0,\pm1,\pm\sqrt{2}, \frac{-1\pm\sqrt{5}}{2}, \frac{1\pm\sqrt{5}}{2}\right\}.$$  However, $0, \pm 1,$ and $\pm \sqrt{2}$ do not belong to the $5$-adic filled Julia set of $h^{\phi}$. It is easy to compute that $\frac{-1\pm\sqrt{5}}{2}$ and $\frac{1\pm \sqrt{5}}{2}$ are $h^{\phi}$-preperiodic.
			Applying Proposition \ref{changeofcoordinates}(i), the desired result follows.
	\end{proof}
\section{Totally $p$-adic preperiodic points}
Fix an odd prime $p$.  In \cite[Th$\acute{\text{e}}$or$\grave{\text{e}}$me 1.1]{BBP}, Benedetto-Briend-Perdry  provided a  trichotomy of the $p$-adic filled Julia set. The trichotomy plays a crucial role in understanding totally $p$-adic preperiodic points of $f_c$ in Theorem  \ref{totpadictri}.
\begin{theorem}\label{BBPmainthm} \cite{BBP} Let $p$ be an odd prime and let $\phi(x)=x^2+c$ be a quadratic polynomial defined on $\mathbb{Q}_p$. For $c\in\mathbb{Q}_p$, we have
	\begin{enumerate}
		\item[(i)] If $|c|_p\leq 1$, then $\mathcal{K}_{\mathbb{Q}_p}(f_c)=\{x\in\mathbb{Q}_p \mid |x|_p\leq 1\}$.
		\item[(ii)] If $|c|_p>1$ and $-c$ is not a square in $\mathbb{Q}_p$, then $\mathcal{K}_{\mathbb{Q}_p}(f_c)$ is empty.
		\item[(iii)] If $|c|_p>1$ and $-c$ is a square in $\mathbb{Q}_p$, then $\mathcal{K}_{\mathbb{Q}_p}(f_c)$ is nonempty and is homeomorphic to a Cantor set. Moreover,  $\mathcal{K}_{\mathbb{C}_p}(f_c)=\mathcal{K}_{\mathbb{Q}_p}(f_c)$.
	\end{enumerate}
\end{theorem}
It is known that the transfinite diameter of the ring of integers $\mathbb{Z}_p$ in $\mathbb{Q}_p$ is  less than $1$ and it is computed precisely in \cite[Example 5.2.17]{Rum}
$$d_{\infty}(\mathbb{Z}_p)=p^{-\frac{1}{p-1}}.$$ Combining this fact, Proposition \ref{adelicfekete}, and Theorem \ref{BBPmainthm}, it gives rise to the following trichotomy of the set of totally $p$-adic $f_c$-preperiodic points. Recall that an algebraic number $\alpha$ is totally $p$-adic if its minimal polynomial over $\mathbb{Q}$ has all roots embedded into  $\mathbb{Q}_p.$
\begin{theorem} \label{totpadictri} Fix an odd prime $p$. Let $c$ be a rational number.  
	\begin{enumerate}
		\item[(a)] If $|c|_p\leq 1$, then $\mathrm{PrePer}(f_c)\cap \mathbb{Q}^{\mathrm{tp}}$ is  finite.
		\item[(b)] If $|c|_p>1$  and $-c$ is not a square in $\mathbb{Q}_p$, then $\mathrm{PrePer}(f_c)\cap \mathbb{Q}^{\mathrm{tp}}$ is empty.
		\item[(c)] If $|c|_p>1$  and $-c$ is a square in $\mathbb{Q}_p$, then $\mathrm{PrePer}(f_c)\subset \mathbb{Q}^{\mathrm{tp}}$; in particular, $\mathrm{PrePer}(f_c)$ contains infinitely many totally $p$-adic $f_c$-preperiodic points. 
	\end{enumerate}
\end{theorem}
\begin{proof} Let $|c|_p\leq 1$. Then we consider the adelic set $\mathbb{E}=\{E_v\}_v$ where 
	 $$E_v=\begin{cases}\mathcal{K}_{\mathbb{C}_q}(f_c)&\mbox{if}\,\, v=q\neq p\\ \mathbb{Z}_p&\mbox{if}\,\, v=p\\ \mathcal{K}_{\mathbb{C}}(f_c)&\mbox{if}\,\, v=\infty\end{cases}$$
	Hence the transfinite diameter of the adelic set $\mathbb{E}$ is 
	\begin{align*}d_{\infty}(\mathbb{E})&=d_{\infty}(\mathbb{Z}_p)\prod_{v\neq p} d_{\infty}(\mathcal{K}_{\mathbb{C}_v}(f_c))\\&=d_{\infty}(\mathbb{Z}_p) \quad(\text{using}\,\,\, (\ref{capvadicfilled}))\\&=p^{-\frac{1}{p-1}}<1.\end{align*}
	Applying Proposition \ref{adelicfekete}, it yields that  there are only finitely many algebraic numbers all of whose algebraic conjugates lie  in $\mathcal{K}_{\mathbb{C}_q}(f_c)$ at each place $q\neq p$, in $\mathbb{Z}_p$, and in $\mathcal{K}_{\mathbb{C}}(f_c)$ at the archimedean place $\infty$. The finiteness of $\mathrm{PrePer}(f_c)\cap \mathbb{Q}^{\mathrm{tp}}$ follows.  This completes part (a).
	The proof of  (b)  follows immediately from Theorem \ref{BBPmainthm} (ii). It remains to verify (c). We apply  Theorem \ref{BBPmainthm} (iii) to obtain
	$$ \mathcal{K}_{\mathbb{C}_p}(f_c)=\mathcal{K}_{\mathbb{Q}_p}(f_c)\subseteq \mathbb{Q}_p.$$
	If $\alpha$ is $f_c$-preperiodic, then so are all of its algebraic conjugates, and therefore all preperiodic points are totally $p$-adic.
	Hence $\mathrm{PrePer}(f_c)\cap\mathbb{Q}^{\mathrm{tp}}$ is infinite.
\end{proof}
\begin{remark} We do not completely know if the set $\mathrm{PrePer}(f_c)\cap \mathbb{Q}^{\mathrm{tp}}$ is always nonempty when $|c|_p\leq 1$. However, $f_c$ has a totally $p$-adic fixed point when $1-4c$ is a square in $\mathbb{Q}_p$ by applying  Hensel's lemma to  the polynomial $f(X)=X^2-(1-4c)\in\mathbb{Z}_p[X]$.
\end{remark}

\end{document}